\documentclass{article}[12pt]
\usepackage{graphicx}
\usepackage[english]{babel}
\usepackage[autostyle]{csquotes}
\usepackage{amsmath,amsfonts}
\usepackage{amsthm}

\usepackage{geometry}
\usepackage[toc,page]{appendix}
\usepackage{subcaption}
\usepackage{hyperref}
\usepackage{setspace}
\hypersetup{
	colorlinks=true,
	linkcolor=blue,
	filecolor=blue,      
	urlcolor=blue,
	citecolor = blue,
	anchorcolor = blue 
}
\newtheorem{theorem}{Theorem}[section]

\newtheorem{lemma}[theorem]{Lemma}

\newtheorem{assumption}{Assumption}

\providecommand{\keywords}[1]{\textbf{\textbf{Key Words:}} #1}

\begin{document}

\title{\bf Existence and Consistency of the Maximum Pseudo $\beta$-Likelihood Estimators for Multivariate Normal Mixture Models}
\author{Soumya Chakraborty\footnote{Joint Affiliation at Bethune College (Assistant Professor) and Indian Statistical Institute (Research Fellow)}, Ayanendranath Basu\footnote{ayanbasu@isical.ac.in (Corresponding Author)} and Abhik Ghosh
\\
Indian Statistical Institute, Kolkata, India.   
}

\maketitle

\begin{abstract}
Robust estimation under multivariate normal (MVN) mixture model is always a computational challenge. A recently proposed maximum pseudo $\beta$-likelihood estimator aims to estimate the unknown parameters of a MVN mixture model in the spirit of minimum density power divergence (DPD) methodology but with a relatively simpler and tractable computational algorithm even for larger dimensions. In this letter, we will rigorously derive the existence and weak consistency of the maximum pseudo $\beta$-likelihood estimator in case of MVN mixture models under a reasonable set of assumptions. Two real data examples are also presented in order to assess the practical utility of our method.
\\
\\

\end{abstract}
\section{Introduction}
Mixture models arise in many practical scenarios when the variable of interest depends on certain categorical attributes. Finite mixture models with various shapes (for the individual components) have been introduced and extensively studied in the statistical literature (see, e.g., McLachlan and Peel $(2004)$ \cite{peel}). Normal mixture models are perhaps the most flexible among the introduced mixture probability distributions which can fit a large variety of datasets. In the era of big data, practical datasets contain information about different attributes, and thus, they are naturally multivariate. Hence, multivariate normal mixture models are of great importance for modelling these datasets. Some well-known practical applications of mixture models (including the normal mixture model) include clustering, voice recognition, medical image reconstruction, handwriting discrimination and much more. 

Mathematically, a $p$-dimensional normal mixture distribution with $k$ components has the probability density function (PDF) as given by
\begin{align}
\label{nmmden}
f_{\boldsymbol{\theta}}(\boldsymbol{x})=\sum_{j=1}^{k}\pi_{j}\phi_{p}(\boldsymbol{x},\boldsymbol{\mu}_{j},\boldsymbol{\Sigma}_{j}),
\end{align}
where $\phi_{p}(\cdot,\boldsymbol{\mu},\boldsymbol{\Sigma})$ denotes the PDF of a $p$-dimensional multivariate normal (MVN) denisty with mean $\boldsymbol{\mu}$ and dispersion matrix $\boldsymbol{\Sigma}$, $\boldsymbol{\mu}_{j} \in  \mathbb{R}^{p}$ is the $j$-th component mean, $\boldsymbol{\Sigma}_{j}$ is the $j$-th component dispersion matrix (which is assumed to be a $p$-dimensional real, symmetric and positive definite matrix) and $0 \leq \pi_{j} \leq 1$ is the $j$-th component weight, for $ j=1,2,...,k$, with $\sum_{j=1}^{k}\pi_{j}=1$. 

It is essential for the analyses of multivariate data with MVN mixture model to appropriately estimate the unknown component parameters, i.e., ($\mu_j$, $\Sigma_j$, $\pi_j$) for all $j=1, \ldots, k$, where $k$ is usually a pre-fixed number of components in the mixture modeling (note that, we can also choose $k$ based on the observed data through some data-driven algorithms, but we will not consider that aspect in this letter). Maximum likelihood estimators of these unknown parameters are not straightforward to obtain as closed form solution of the corresponding estimating (score) equations do not exist. EM type algorithms are generally used to solve the estimating equations. But these estimators are highly non-robust in the presence of contamination in the sample data. Thus, robust estimation of the parameters in a MVN mixture model is important in order to produce stable results and inference. Various robust methodologies have been derived for normal mixture models following different philosophies, eg., trimming (Cuesta-Albertos et al.~$(1997)$ \cite{trimmedkmeans}, Garc\'{i}a-Escudero et al.~$(2008)$ \cite{tclust}), divergence minimization (Fujisawa and Eguchi $(2006)$ \cite{dpdjapan}) and others. Chakraborty et al.~$(2022)$ \cite{arxiv} have recently proposed a robust procedure (known as the maximum pseudo $\beta$-likelihood estimation) which performs robust estimation of the  component parameters, along with data clustering and anomaly detection in the spirit of minimum density power divergence (DPD) estimation principle (Basu et al.~$1998$ \cite{basupaper}, $2011$ \cite{basubook}). These authors have argued that the original minimum DPD estimator under the Gaussian mixture models, as studied by Fujisawa and Eguchi $(2006)$ \cite{dpdjapan}, is difficult to implement in case of multivariate data with higher dimensions whereas the maximum pseudo $\beta$-likelihood estimator (MPLE$_\beta$) can easily be computed but continues to have high efficiency under pure and strong stability under contaminated multivariate data. However, Chakraborty et al. (2022) \cite{arxiv} presented the complete methodology of their MPLE$_\beta$ algorithm, the asymptotic properties of these estimators have not been discussed. 

This article aims to establish the existence and consistency of the MPLE$_\beta$ in case of the MVN mixture models under a set of reasonable sufficient conditions. Modern empirical process arguments are utilized to establish consistency of the estimators. Section \ref{sec2} presents a brief description of maximum pseudo $\beta$-likelihood estimation. Section \ref{SEC:theoretical_result} provides the detailed proofs of the two main results on existence and consistency. The application of our method is illustrated in Section \ref{SLC} on two real data sets. Concluding remarks are given in Section \ref{con}. The Appendix contains a few relevant lemmas. 
\section{Maximum Pseudo $\beta$-Likelihood Estimation}
\label{sec2}
 Let, $\boldsymbol{X}_{1}, \boldsymbol{X}_{2},...,\boldsymbol{X}_{n}$ be a random sample drawn from a $p$-dimensional MVN mixture distribution with $k$ components and the unknown PDF of this random sample is modelled by the family $\{f_{\boldsymbol{\theta}}:\boldsymbol{\theta} \in \boldsymbol{\Theta}\}$ with $f_{\boldsymbol{\theta}}$ as in Equation (\ref{nmmden}). The whole parameter $\boldsymbol{\theta}$ is given by $\boldsymbol{\theta} = (\pi_{1}, \pi_{2},...., \pi_{k}, \boldsymbol{\mu}_{1}, \boldsymbol{\mu}_{2},...,
 \boldsymbol{\mu}_{k}, \boldsymbol{\Sigma}_{1}, \boldsymbol{\Sigma}_{2}\\,..., \boldsymbol{\Sigma}_{k})$. The robust MPLE$_\beta$ of $\boldsymbol{\theta}$ is defined as the maximizer of the empirical pseudo $\beta$-likelihood function given by (see Chakraborty et al.~$(2022)$ \cite{arxiv} for more details) 
\begin{equation}
\centering
\label{Eq4}
L_{\beta}(\boldsymbol{\theta},F_{n})=E_{F_{n}}\left[\sum_{j=1}^{k}Z_{j}(\boldsymbol{X},\boldsymbol{\theta})\left[\log\;\pi_{j}+\frac{1}{\beta}\phi_{p}^{\beta}(\boldsymbol{X},\boldsymbol{\mu}_{j},\boldsymbol{\Sigma}_{j})-\frac{1}{1+\beta}\int \phi_{p}^{1+\beta}(\boldsymbol{x},\boldsymbol{\mu}_{j},\boldsymbol{\Sigma}_{j}) \;d\boldsymbol{x}\right]\right],
\end{equation}
where $F_{n}$ is the empirical cumulative distribution function based on the random sample $\boldsymbol{X}_{1}, \boldsymbol{X}_{2},...,\boldsymbol{X}_{n}$, and $Z_{j}(\cdot,\boldsymbol{\theta})$, $j=1,\ldots,\;k$, are the assignment functions given by the indicators  $Z_{j}(\boldsymbol{X}_{i},\boldsymbol{\theta})=I[D(\boldsymbol{X}_{i},\boldsymbol{\theta})\\=D_{j}(\boldsymbol{X}_{i},\boldsymbol{\theta})]$ with $ D_{j}(\boldsymbol{X},\boldsymbol{\theta}) = \pi_{j}\phi_{p}(\boldsymbol{X},\boldsymbol{\mu}_{j},\boldsymbol{\Sigma}_{j})  \;\; \text{and}\;D(\boldsymbol{X},\boldsymbol{\theta}) =\underset{1 \leq j \leq k}{\text{max}}D_{j}(\boldsymbol{X},\boldsymbol{\theta})$. Here, $\beta$ is a tuning parameter that controls the trade-off between robustness and asymptotic efficiency of the resulting estimators.  

Assuming $F$ to be the cumulative distribution function (true and unknown) of $\boldsymbol{X}_{1}$, the corresponding theoretical objective function (the population pseudo $\beta$-likelihood) is given by,
\begin{equation}
\centering
\label{Eq5}
L_{\beta}(\boldsymbol{\theta},F)=E_{F}\left[\sum_{j=1}^{k}Z_{j}(\boldsymbol{X},\boldsymbol{\theta})\left[\log\;\pi_{j}+\frac{1}{\beta}\phi_{p}^{\beta}(\boldsymbol{X},\boldsymbol{\mu}_{j},\boldsymbol{\Sigma}_{j})-\frac{1}{1+\beta}\int \phi_{p}^{1+\beta}(\boldsymbol{x},\boldsymbol{\mu}_{j},\boldsymbol{\Sigma}_{j}) \;d\boldsymbol{x}\right]\right].
 \end{equation}
 For technical reasons (see Chakraborty et al.~$(2022)$ \cite{arxiv}), the empirical objective function is maximized subject to the following pair of constraints, namely, the eigenvalue ratio (ER) and non-singularity (NS) constraints. Here, we assume $\lambda_{jl}$ to be the $l$-th eigenvalue of the covariance matrix $\boldsymbol{\Sigma}_{j}$ for $1 \leq j \leq k$ and $1 \leq l \leq p$,
 and put  $M=\underset{1 \leq j \leq k}{\text{max}}\;\underset{1 \leq l \leq p}{\text{max}}\lambda_{jl}$ and 
 $m=\underset{1 \leq j \leq k}{\text{min}}\;\underset{1 \leq l \leq p}{\text{min}}\lambda_{jl}$, the largest and smallest eigenvalues, respectively.
 \begin{itemize}
 
 	\item \noindent\textbf{Eigenvalue Ratio (ER) Constraint:}  For a prespecified constant $c \geq 1$, the component dispersion matrices, namely, $\boldsymbol{\Sigma}_{1}, \boldsymbol{\Sigma}_{2},..., \boldsymbol{\Sigma}_{k}$, satisfy
 	\begin{equation}
 	\centering
 	\label{Eq6}
 	\frac{M}{m} \leq c. 
 	\end{equation}
 	\item \noindent \textbf{Non-singularity (NS) Constraint:} We assume that the smallest eigenvalue $m$ satisfies
 	$
 	m\geq c_{1}
 	$
 	for some small positive constant $c_{1}$ which is prespecified.
 \end{itemize}
 Under the above two constraints, characterized by constants  $C=(c,c_{1})$, 
 our search for the estimators can be confined with the restricted parameter space defined as
 $$
 \boldsymbol{\Theta}_{C}=\left\{\boldsymbol{\theta}:\boldsymbol{\theta}= (\pi_{1}, \pi_{2},...., \pi_{k}, \boldsymbol{\mu}_{1}, \boldsymbol{\mu}_{2},...,\boldsymbol{\mu}_{k}, \boldsymbol{\Sigma}_{1}, \boldsymbol{\Sigma}_{2},..., \boldsymbol{\Sigma}_{k})\;\text{with}\; \frac{M}{m} \leq c \mbox{     and } m\geq c_{1} \right\}.
 $$
Although, the methodology (and the influence function analysis) of the resulting estimators have been extensively discussed in Chakraborty et al.~$(2022)$ \cite{arxiv}, the asymptotic properties of the same remain to be established. In this article, we are going to prove the existence (of the optimizer of the aforesaid constrained optimization problem) and the weak consistency of the resulting estimators in the following section.  
\section{Theoretical Results}
\label{SEC:theoretical_result}
Let us first state the following technical assumption which will be crucial to prove some of our results.
\begin{assumption}
	\label{assump1}
	For the true (unknown) distribution function $F=\sum_{j=1}^{k}\pi^{0}_{j}\phi_{p}(\cdot,\boldsymbol{\mu^{0}_{j}},\boldsymbol{\Sigma^{0}_{j}})$, let us assume that, $\underset{1\leq j \leq k}{max}\pi^{0}_{j}\geq (1+k^{0})\frac{\beta}{(1+\beta)^{1+\frac{p}{2}}}$ for some $k^{0}>0$.
	
\end{assumption}

Firstly, we present the result on the existence of the optimizers (both empirical and population versions) and then the consistency of the resulting estimators under the MVN mixture models. 



\begin{theorem}[Existence]\label{THM:Existence}
Let $P$ be a  probability distribution (either $F$ or $F_{n}$ in our case) and Assumption \ref{assump1} is satisfied by $F$. 
Then, there exists  $\boldsymbol{\theta} \in \boldsymbol{\Theta}_{C}$ that maximizes
	\begin{align*}
	 L_{\beta}(\boldsymbol{\theta},P)=E_{P}\left[\sum_{j=1}^{k}Z_{j}(\boldsymbol{X},\boldsymbol{\theta})\left[\log\;\pi_{j}+\frac{1}{\beta}\phi_{p}^{\beta}(\boldsymbol{X},\boldsymbol{\mu}_{j},\boldsymbol{\Sigma}_{j})-\frac{1}{1+\beta}\int \phi_{p}^{1+\beta}(\boldsymbol{x},\boldsymbol{\mu}_{j},\boldsymbol{\Sigma}_{j}) \;d\boldsymbol{x}\right]\right]
	\end{align*}
under the ER and NS constraints (for a sufficiently large sample size $n$ in case of $F_n$).
\end{theorem}
\begin{proof}

Let $\{\boldsymbol{\theta}^{r}\}=\{(\pi^{r}_{1}, \pi^{r}_{2},...., \pi^{r}_{k}, \boldsymbol{\mu}^{r}_{1}, \boldsymbol{\mu}^{r}_{2},...,\boldsymbol{\mu}^{r}_{k}, \boldsymbol{\Sigma}^{r}_{1}, \boldsymbol{\Sigma}^{r}_{2},..., \boldsymbol{\Sigma}^{r}_{k})\}$ be a sequence in $\boldsymbol{\Theta}_{C}$ such that,
\begin{equation}
\centering
\label{Eq17}
\underset{r\rightarrow\infty}{\text{lim}}L_{\beta}(\boldsymbol{\theta}^{r},P)=\underset{\boldsymbol{\theta} \in \boldsymbol{\Theta}_{C}}{\sup}L_{\beta}(\boldsymbol{\theta},P).
\end{equation}
Let us assume, without loss of generality, that, $\pi^{0}_{1}=\underset{1\leq j \leq k}{max}\;\pi^{0}_{j}$.
Let, $\boldsymbol{\theta}^{a}= (\pi^{a}_{1}, \pi^{a}_{2},...., \pi^{a}_{k}, \boldsymbol{\mu}^{a}_{1}, \boldsymbol{\mu}^{a}_{2},...,\boldsymbol{\mu}^{a}_{k},\\ \boldsymbol{\Sigma}^{a}_{1}, \boldsymbol{\Sigma}^{a}_{2},..., \boldsymbol{\Sigma}^{a}_{k}) \in \boldsymbol{\Theta}_{C}$ such that,
\begin{align*}
\pi^{a}_{j}=1\;,\;\boldsymbol{\mu}^{a}_{j}=\boldsymbol{\mu}^{0}_{j} ,\;\boldsymbol{\Sigma}^{a}_{j}=c^{'}\boldsymbol{\Sigma}^{0}_{j}\; \text{for}\;j=1.
\end{align*}
This implies that,
\[
Z_{j}(\boldsymbol{X},\boldsymbol{\theta}^{a})= 
\begin{cases}
1,& \text{if } j= 1\\
0,              & \text{if }  2 \leq j\leq k.
\end{cases}
\]
Hence,
\begin{align*}
\underset{r\rightarrow\infty}{\text{lim}}L_{\beta}(\boldsymbol{\theta}^{r},P)=\underset{\boldsymbol{\theta} \in \boldsymbol{\Theta}_{C}}{\sup}L_{\beta}(\boldsymbol{\theta},P)\geq L_{\beta}(\boldsymbol{\theta}^{a},P).
\end{align*}
If $P$ is the true unknown distribution function (i.e., $F$), then,
\begin{align*}
L_{\beta}(\boldsymbol{\theta}^{a},P)&=E_{F}\left(\frac{1}{\beta}\phi^{1+\beta}_{p}(\boldsymbol{X},\boldsymbol{\mu}^{0}_{1},c^{'}\boldsymbol{\Sigma}^{0}_{1})\right)-\frac{1}{1+\beta}\int \phi^{1+\beta}_{p}(\boldsymbol{x},\boldsymbol{\mu}^{0}_{1},c^{'}\boldsymbol{\Sigma}^{0}_{1})\;d\boldsymbol{x}\\
&\geq \pi^{0}_{1}E_{F_{1}}\left(\frac{1}{\beta}\phi^{1+\beta}_{p}(\boldsymbol{X},\boldsymbol{\mu}^{0}_{1},c^{'}\boldsymbol{\Sigma}^{0}_{1})\right)-\frac{1}{1+\beta}\int \phi^{1+\beta}_{p}(\boldsymbol{x},\boldsymbol{\mu}^{0}_{1},c^{'}\boldsymbol{\Sigma}^{0}_{1})\;d\boldsymbol{x}\\
&=\frac{1}{\beta(2\pi)^{\frac{p\beta}{2}}|c^{'}\Sigma^{0}_{1}|^{\frac{\beta}{2}}(1+\dfrac{\beta}{c^{'}})^{\frac{p}{2}}}\left[\pi^{0}_{1}-\frac{\beta}{(1+\beta)^{1+\frac{p}{2}}}\left(1+\frac{\beta}{c^{'}}\right)^\frac{p}{2}\right]\;\text{(after some algebra)}.
\end{align*}
Now, the positivity of the aforesaid term can be achieved by taking a large enough $c^{'}$ (as a consequence of Assumption \ref{assump1}), and thus 
\begin{align*}
L_{\beta}(\boldsymbol{\theta}^{a},P)\geq 0.
\end{align*}
But if $P$ is the empirical distribution $F_{n}$ ($n$ represents the sample size),
\begin{align*}
L_{\beta}(\boldsymbol{\theta}^{a},P)=\frac{1}{n\beta}\sum_{i=1}^{n}\phi^{\beta}_{p}(\boldsymbol{X}_{i},\boldsymbol{\mu^{0}_{1}},c^{'}\boldsymbol{\Sigma^{0}_{1}})-\frac{1}{1+\beta}\int \phi^{1+\beta}_{p}(\boldsymbol{x},\boldsymbol{\mu^{0}_{1}},c^{'}\boldsymbol{\Sigma^{0}_{1}})\;d\boldsymbol{x}.
\end{align*}
The positivity of the above quantity can be easily established by an application of the strong law of large numbers (SLLN) assuming a moderately large sample size $n$ followed by the aforesaid argument to prove the positivity in case of $P=F$.
So, the sequence $\{\boldsymbol{\theta}^{r}\}$ satisfies,
\begin{equation}
\centering 
\label{Eq18}
\underset{r\rightarrow\infty}{\text{lim}}L_{\beta}(\boldsymbol{\theta}^{r},P)\geq 0.
\end{equation}
Since $(\pi^{r}_{1}, \pi^{r}_{2},...., \pi^{r}_{k}) \in [0,1]^{k}$ and $[0,1]^{k}$ is a compact set in $\mathbb{R}^{k}$, the sequence $\{\boldsymbol{\theta}^{r}\}$ has a subsequence $\{\boldsymbol{\theta}^{r}\}^{l}$ such that $\{\pi^{r}_{1},\pi^{r}_{2},...,\pi^{r}_{k}\}^{l}$ is convergent. To simplify the notation, we will denote this subsequence $\{\boldsymbol{\theta}^{r}\}^{l}$  as the original sequence $\{\boldsymbol{\theta}^{r}\}$.

Hence the sequence must satisfy the following properties.
\begin{enumerate}
	\item For the proportion sequence $\{\pi^{r}_{1},\pi^{r}_{2},...,\pi^{r}_{k}\}$,
	\begin{align}
	\label{Eq19}
	\pi^{r}_{j}&\rightarrow\pi_{j} \in [0,1]\; \text{for}\;1\leq j\leq k.
	\end{align}
	\item For the mean sequence $\{\boldsymbol{\mu}^{r}_{1},\boldsymbol{\mu}^{r}_{2},...,\boldsymbol{\mu}^{r}_{k}\}$,
	\begin{align}
	\label{Eq20}
	\boldsymbol{\mu}^{r}_{j}&\rightarrow\boldsymbol{\mu}_{j} \in \mathbb{R}^{p}\; \text{for}\; j=1,2,...,g\;\text{and} \;||\boldsymbol{\mu}^{r}_{j}||\rightarrow \infty\; \text{for}\;j=g+1,...,k\;\text{for}\;\text{some}\; 0 \leq g\leq k.
	\end{align}
	\item Finally, the dispersion sequence $\{\boldsymbol{\Sigma}^{r}_{1},\boldsymbol{\Sigma}^{r}_{2},...,\boldsymbol{\Sigma}^{r}_{k}\}$ must satisfy exactly one of the following conditions.
	Either,
	\begin{align}
	\label{Eq21}
	\boldsymbol{\Sigma}^{r}_{j}&\rightarrow\boldsymbol{\Sigma}_{j}\in \mathbb{R}^{p \times p} \; \text{for}\;1\leq j\leq k,
	\end{align}
	or,
	\begin{align}
	\label{Eq22}
	M_{r} \rightarrow \infty,
	\end{align}
	or,
	\begin{align}
	\label{Eq23}
	m_{r} \rightarrow 0,
	\end{align}
	where $M_r$ and $m_r$ are the largest and the smallest elements of the set of eigenvalues of $\boldsymbol{\Sigma}^{r}_{1},\boldsymbol{\Sigma}^{r}_{2},...,\boldsymbol{\Sigma}^{r}_{k}$, respectively.
\end{enumerate}
Now, by Lemma \ref{baz1}, presented in Appendix \ref{appenA}, Equation (\ref{Eq20}) holds for $g=k$ in case of component means and Equation (\ref{Eq21}) holds for the component covariance matrices. Thus, if $\pi_j >0$ for all $j=1, 2, \ldots, k$ in Equation (\ref{Eq19}), then, the choice of the optimizer is obvious. But, if $\pi_{j}>0$ for $j=1,2,\ldots,g$ for some $1\leq g <k$, and $\pi_{j}=0$ for $j>g$, then take $\pi_{j}=\underset{r\rightarrow \infty}{lim}\; \pi^{r}_{j}$ for $j=1,2,\ldots,g$ and $\pi_{j}=0$ for $j>g$, $\mu_{j}=\underset{r\rightarrow \infty}{lim}\; \mu^{r}_{j}$, $\Sigma_{j}=\underset{r\rightarrow \infty}{lim}\; \Sigma^{r}_{j}$ for $j=1,2,\ldots,g$ and $\mu_{j}$ and $\Sigma_{j}$ arbitrarily (satisfying the ER and NS constraints) for $j>g$. These values will provide the maximizer of the objective function in consideration. This completes the proof of Theorem \ref{THM:Existence}.
\end{proof}

Our next theorem provides consistency properties of the MPLE$_{\beta}$ (that is, the consistency of the sample version (optimizer in Equation (\ref{Eq4})) to the population version (optimizer in Equation (\ref{Eq5}))). To achieve this, we need uniqueness of the maximizer of (\ref{Eq5}) under the ER and NS constraints and the following prerequisite result in order to establish the consistency.

\begin{theorem}[Corollary $3.2.3$, van der Vaart and Wellner $(1996)$ \cite{vandervaart}]\label{baz2}
	Let $M_{n}$ be a stochastic process indexed by a metric space $\boldsymbol{\Theta}$ and let $M:\boldsymbol{\Theta}\rightarrow\mathbb{R}$ be a deterministic function. Suppose the following conditions hold.
	\begin{enumerate}
		\item Suppose that $||M_{n}-M||_{\boldsymbol{\Theta}}\rightarrow 0$ in probability.
		\item There exists a point $\boldsymbol{\theta}_{0}$ such that, $M(\boldsymbol{\theta}_{0})>\underset{\boldsymbol{\theta} \notin G}{\sup}M(\boldsymbol{\theta})$ for every open set $G$ that contains $\boldsymbol{\theta}_{0}$.
		\item Suppose a sequence $\hat{\boldsymbol{\theta}}_{n}$ satisfies $M_{n}(\hat{\boldsymbol{\theta}}_{n})>\underset{\boldsymbol{\theta}}{\sup}\,M_{n}(\boldsymbol{\theta})-o_{p}(1)$.
	\end{enumerate}
	Then, $\hat{\boldsymbol{\theta}}_{n}\rightarrow\boldsymbol{\theta}_{0}$ in probability.
\end{theorem}

\begin{theorem}[Consistency]\label{THM:Consistency}
Suppose that $\boldsymbol{\theta}_{0}$ be the unique maximizer of (\ref{Eq5}) subject to the ER and NS constraints and let Assumption \ref{assump1} holds. 
Then, if $\hat{\boldsymbol{\theta}}_{n}$ is a maximizer of (\ref{Eq4}) based on a sample of size $n$, we have
$\hat{\boldsymbol{\theta}}_{n}\rightarrow\boldsymbol{\theta}_{0}$ in probability as $n \rightarrow \infty$.
\end{theorem}

\begin{proof}

In order to derive the estimators, we had maximized the objective function $L_{\beta}(\boldsymbol{\theta}, F_{n})$ which is not differentiable with respect to the parameter $\boldsymbol{\theta}$. Hence, the standard Taylor series expansion approach (used to derive the asymptotics of maximum likelihood estimators) may not work for this problem. Thus, we are going to use the modern empirical process tricks (Theorem \ref{baz2}) to establish weak consistency.

Following the notations of Theorem \ref{baz2}, we have, $M_{n}(\boldsymbol{\theta})=L_{\beta}(\boldsymbol{\theta},F_n)=E_{F_{n}}(m_{\boldsymbol{\theta}}(\boldsymbol{X})) \;\text{and}\; M(\boldsymbol{\theta})=E_{F}(m_{\boldsymbol{\theta}}(\boldsymbol{X})$ with 
\begin{align*}
 m_{\boldsymbol{\theta}}(\boldsymbol{X})=\sum_{j=1}^{k}Z_{j}(\boldsymbol{X},\boldsymbol{\theta})\left[\log\;\pi_{j}+\frac{1}{\beta}\phi_{p}^{\beta}(\boldsymbol{X},\boldsymbol{\mu}_{j},\boldsymbol{\Sigma}_{j})-\frac{1}{1+\beta}\int \phi_{p}^{1+\beta}(\boldsymbol{x},\boldsymbol{\mu}_{j},\boldsymbol{\Sigma}_{j}) \;d\boldsymbol{x}\right].
\end{align*}
 The third condition of Theorem \ref{baz2} is satisfied due to Theorem \ref{THM:Existence}. The second condition of Theorem \ref{baz2} is also satisfied due to the assumption of the existence of a unique maximizer of the theoretical objective function in Equation (\ref{Eq5}). Thus, we have to check the first condition only. To verify this, we need a Glivenko-Cantelli (GC) property (van der Vaart and Wellner $(1996)$ \cite{vandervaart}) of the class $\mathcal{F}=\{m_{\boldsymbol{\theta}}(\boldsymbol{X}):\boldsymbol{\theta} \in \boldsymbol{\Theta}_{C}\}.$

To do that, first let us observe the fact (van der Vaart and Wellner $(1996)$ \cite{vandervaart}) that any appropriately measurable Vapnik-$\breve{C}$ervonenkis(VC) class is Glivenko-Cantelli(GC) provided its envelope function is integrable.
Hence, it is enough to show that, $\mathcal{F}$ is VC. But to conclude that $\mathcal{F}$ is GC, the integrability of the envelope function is very crucial. To achieve this integrability, we need the compactness of the parameter space which is established by Lemma \ref{boundedness} in Appendix \ref{appenA}, in the almost sure sense. The significance of Lemma \ref{boundedness} is that the estimators are almost surely included in a compact subset $K$ of the actual parameter space $\boldsymbol{\Theta}_{C}$ for sufficiently large sample sizes. Hence, it is enough to focus on the compact subset $K$ instead of the entire parameter space $\boldsymbol{\Theta}_{C}$ which is unbounded.

To establish that $\mathcal{F}$ is GC, we will follow the methodologies developed in Section $2.6$ of van der Vaart and Wellner $(1996)$ \cite{vandervaart} and Kosorok $(2008)$ \cite{kosorok}.
Let us observe the following facts under the assumption $\boldsymbol{\theta} \in K$.
\begin{itemize}
	\item The functions $(\boldsymbol{X}-\boldsymbol{\mu}_{j})^{'}\boldsymbol{\Sigma}^{-1}_{j}(\boldsymbol{X}-\boldsymbol{\mu}_{j})$ are polynomials of degree 2. Hence, these functions together form a finite dimensional vector space and hence is VC.
	\item The function $\phi(x)=e^{-x}$ is monotone and continuous hence $\phi \circ \mathcal{G}$ is VC if  $\mathcal{G}$ is VC.
	\item The sets $\{Z_{j}(\boldsymbol{X},\boldsymbol{\theta})=1\}$ can be obtained through polynomials of degree 2 and hence is VC. Thus, the functions $\{Z_{j}(\boldsymbol{X},\boldsymbol{\theta})\}$ as indicators of the sets  $\{Z_{j}(\boldsymbol{X},\boldsymbol{\theta})=1\}$ are VC.
	\item Suppose $\mathcal{F}_{1}, \mathcal{F}_{2},...,\mathcal{F}_{k}$ be GC classes of functions on the probability measure $P$ and $\phi$ is a continuous function from $\mathbb{R}^{k}$ to $\mathbb{R}$. Then $\mathcal{H}=\phi(\mathcal{F}_{1}, \mathcal{F}_{2},...,\mathcal{F}_{k})$ is GC on the probability measure $P$ provided that $\mathcal{H}$ has an integrable envelop function (Kosorok $(2006)$ \cite{kosorok}, Wellner $(2012)$ \cite{wellner})).
\end{itemize}
The first observation implies that the collection of functions $(\boldsymbol{X}-\boldsymbol{\mu}_{j})^{'}\boldsymbol{\Sigma}^{-1}_{j}(\boldsymbol{X}-\boldsymbol{\mu}_{j})$ is VC. The second observation implies that $\{\phi_{p}^{\beta}(\boldsymbol{x},\boldsymbol{\mu}_{j},\boldsymbol{\Sigma}_{j})\}$ is VC with an integrable envelope function (because of the compactness of $K$ and boundedness of $e^{-x}$) and hence is GC. Similarly,  the third observation implies that the collection of functions $\{Z_{j}(\boldsymbol{X},\boldsymbol{\theta})\}$ is VC and hence is GC. 

Now the fact that $\mathcal{F}$ is GC is easily followed by the fourth observation and the compactness of $K$. This completes the proof of Theorem \ref{THM:Consistency}.

\end{proof}
\section{Real Data Example}
\label{SLC}
In this section, we are going to illustrate our method using two real life datasets; one of them involves univariate data while the other is a multivariate example.  
\subsection{Univariate Case}
Here we apply our method on the Red blood cell sodium-lithium countertransport (SLC) dataset which has been analyzed in the past by several authors including Dudley et al. $(1991)$ \cite{slc91}, Roeder $(1994)$ \cite{slc94} and Fujisawa and Eguchi $(2006)$ \cite{dpdjapan}. Geneticists are concerned about SLC as it may correlate
blood pressure and hence may be a potential factor behind
hypertension. SLC is also less complicated to assess than blood pressure,
because the latter is a complex quantitative trait influenced
by environmental and genetic factors. The sample size is $190$ and the dataset consists of $3$ genotypes, namely, $A_{1}A_{1}$, $A_{1}A_{2}$, $A_{2}A_{2}$. We analyze these data using the ordinary maximum likelihood approach, the maximum  pseudo $\beta$-likelihood approach, as well as the approach of Fujisawa and Eguchi \cite{dpdjapan}. Figure \ref{table 1} presents the histogram of the original data with different fits overlaid, including the maximum likelihood fit with $3$ clusters, the Fujisawa-Eguchi (FE) fit with $\beta = 0.45$ and $3$ clusters, the maximum pseudo $\beta$-likelihood fit with $\beta = 0.5$ and $3$ clusters and the same with $4$ clusters. The component parameter estimates of the aforesaid models are presented in Table \ref{table 3}.

\begin{figure}[h!]
	\centering
	\begin{tabular}{cc}
		\begin{subfigure}{0.4\textwidth}\centering\includegraphics[width=1\columnwidth]{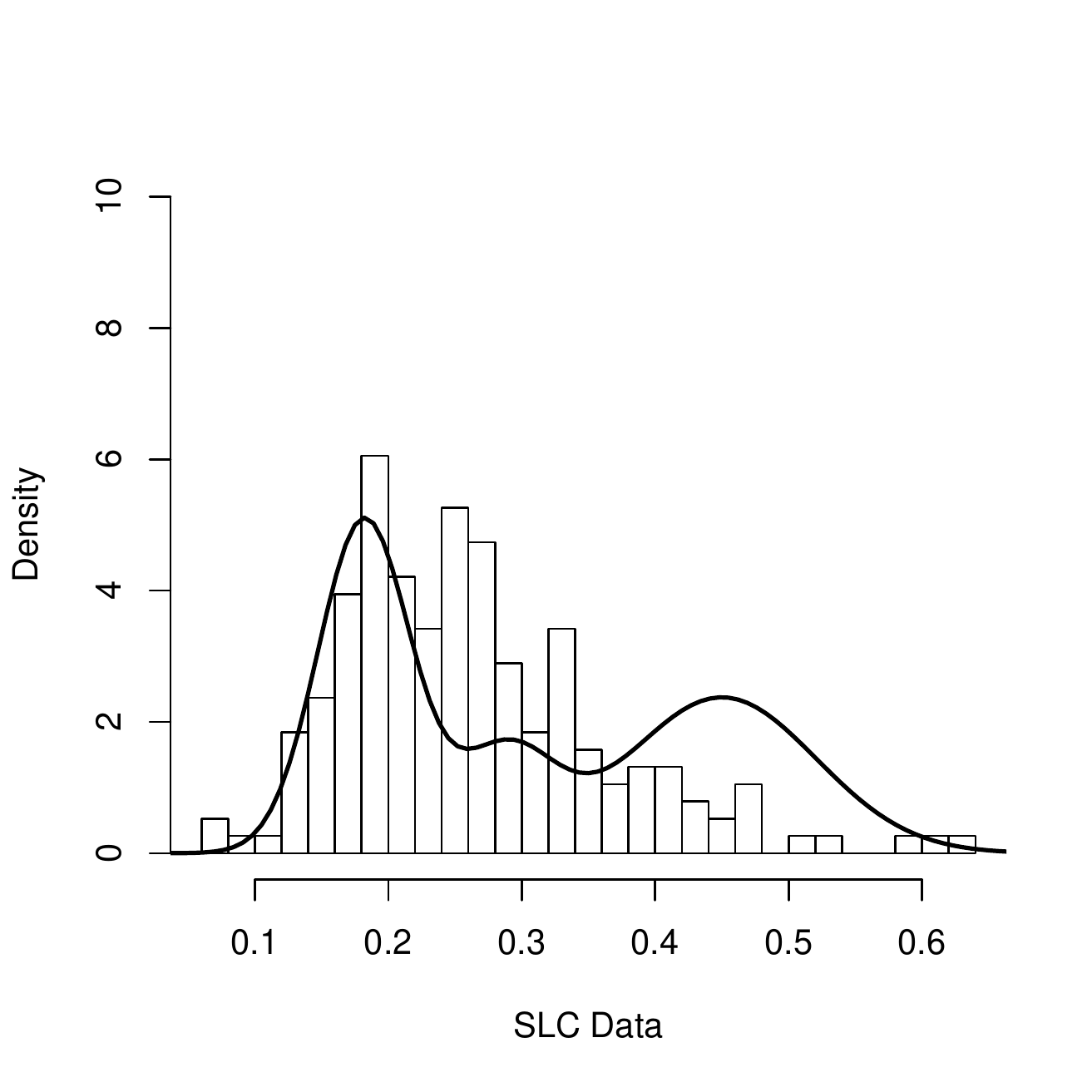}\caption{MLE based fit using $3$ clusters.}\end{subfigure}&
			\begin{subfigure}{0.4\textwidth}\centering\includegraphics[width=1\columnwidth]{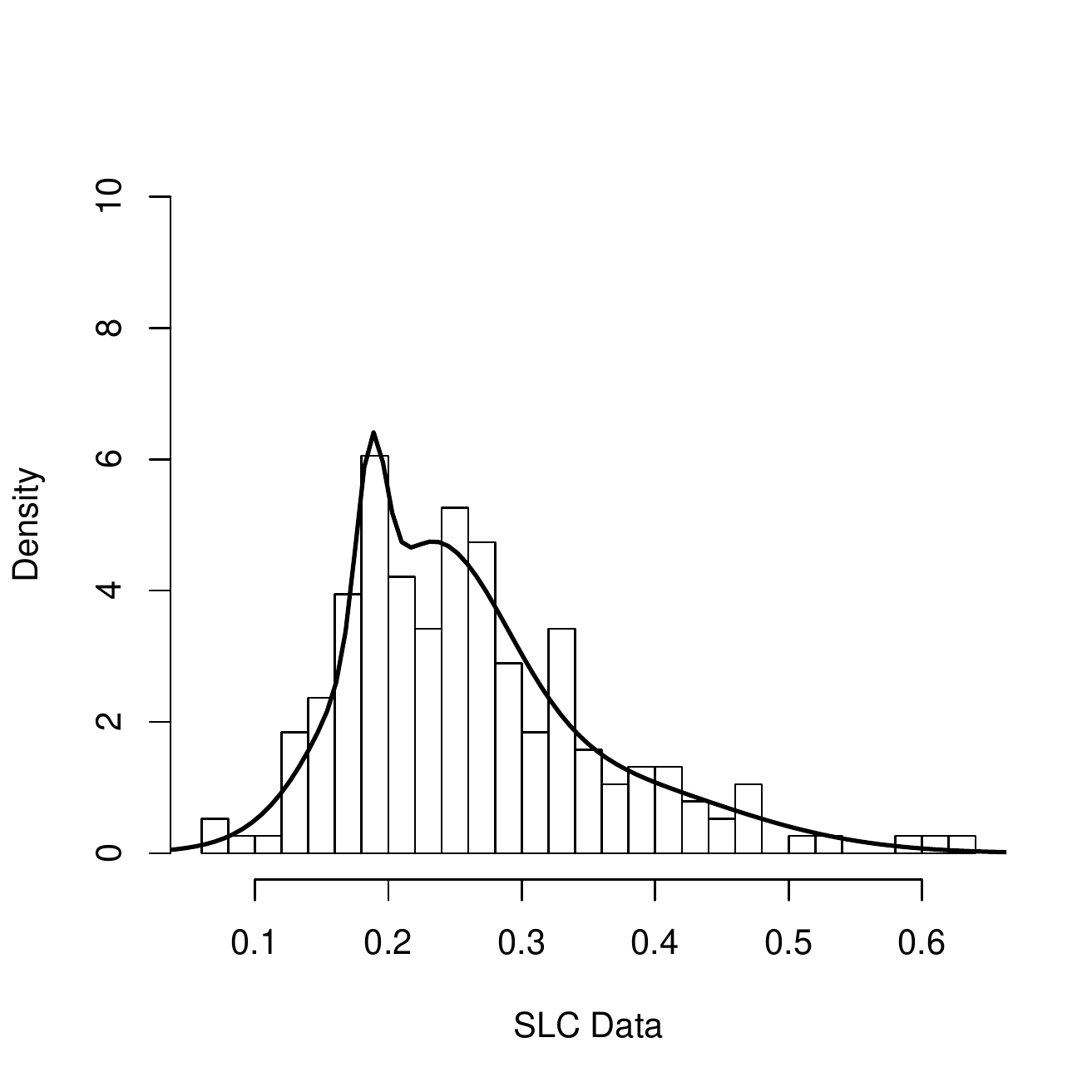}\caption{FE fit using $3$ clusters with $\beta=0.45$.}\end{subfigure}\\
		\begin{subfigure}{0.4\textwidth}\centering\includegraphics[width=1\columnwidth]{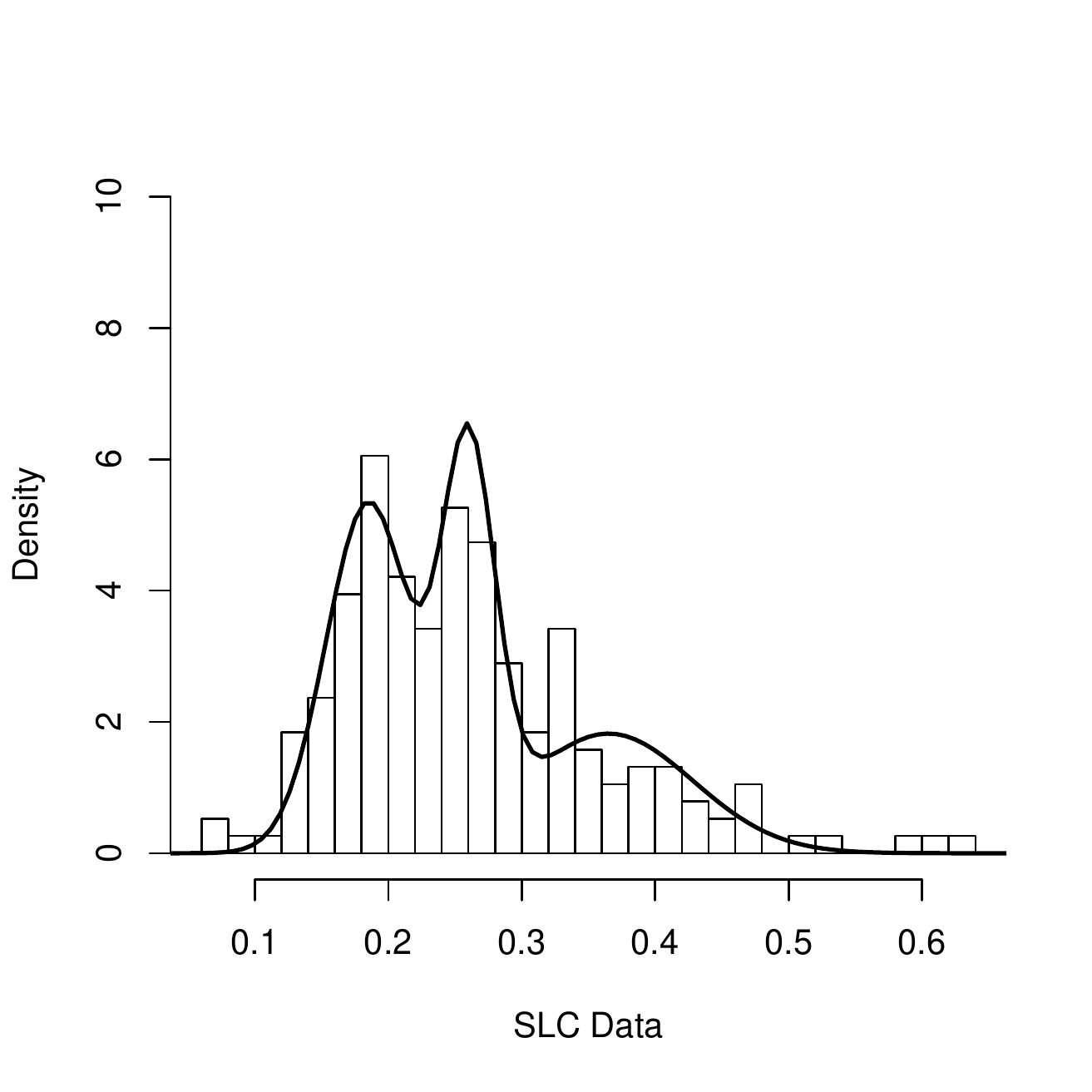}\caption{MPLE$_{\beta}$ fit using $3$ clusters with $\beta=0.5$.}\end{subfigure}&
		\begin{subfigure}{0.4\textwidth}\centering\includegraphics[width=1\columnwidth]{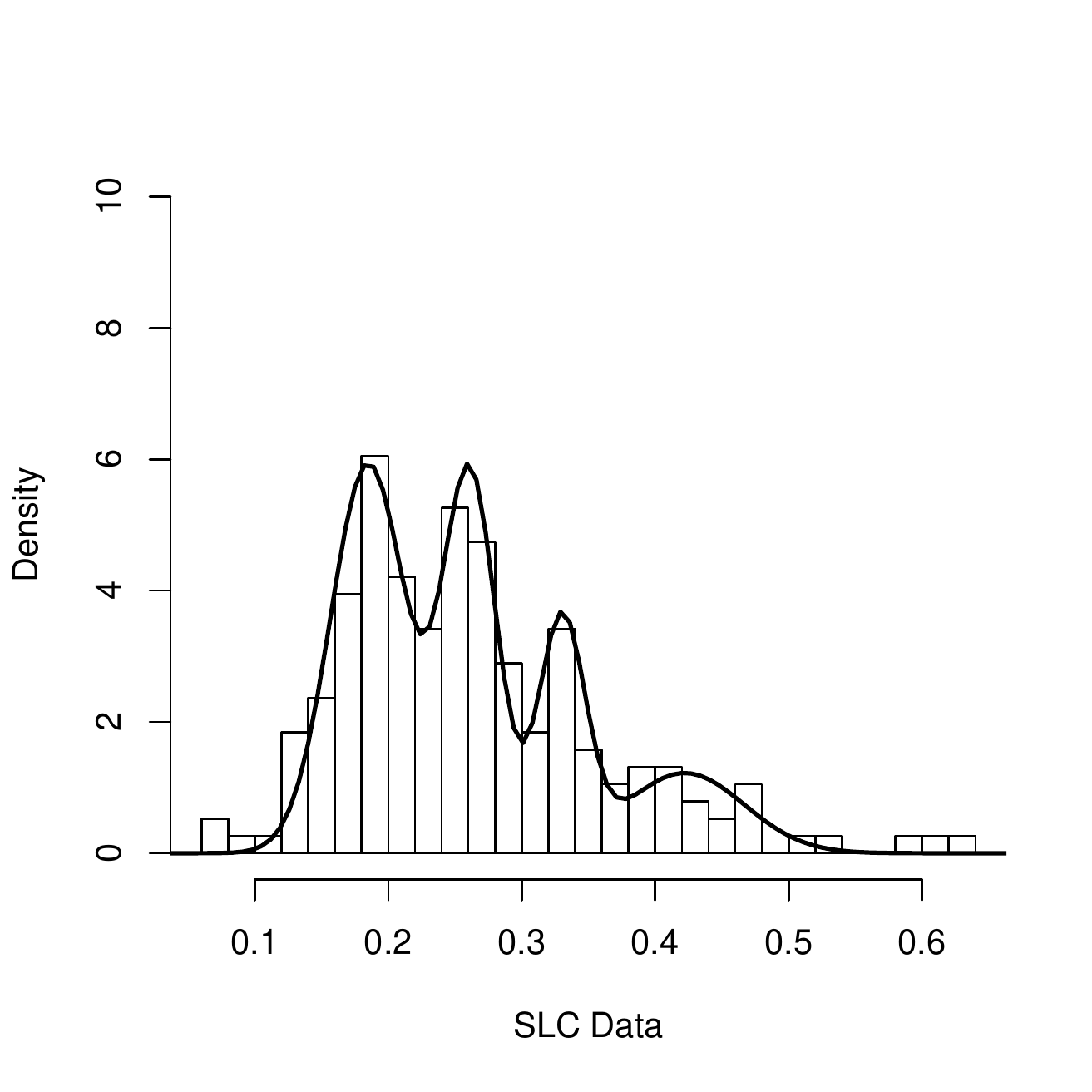}\caption{MPLE$_{\beta}$ fit using $4$ clusters with $\beta=0.5$.}\label{table1.4}\end{subfigure}\\
	\end{tabular}
	\caption{Fitted densities using different methods.}
	\label{table 1}
\end{figure}

\begin{table}[!h]
	\hspace{-2cm}
	\small 
	\begin{tabular}{c c c c c c c c c c c c c c c c c} 
		\hline\\
		\\
		Methods& Clusters & $\hat{\omega}_{1}$ & $\hat{\omega}_{2}$ & $\hat{\omega}_{3}$ & $\hat{\omega}_{4}$  & $\hat{\mu}_{1}$ & $\hat{\mu}_{2}$ & $\hat{\mu}_{3}$  & $\hat{\mu}_{4}$ & $\hat{\sigma}_{1}^{2}$ & $\hat{\sigma}_{2}^{2}$ & $\hat{\sigma}_{3}^{2}$ & $\hat{\sigma}_{4}^{2}$  \\ [1ex] 
		\hline
		MLE & 3 & 0.442 & 0.137 & 0.421 & - & 0.182 & 0.288 & 0.450 &- & 0.001 & 0.001 & 0.005 &- \\ 
	FE $(\beta=0.45)$ & 3 & 0.076 & 0.584 & 0.340 & - & 0.187 & 0.227 & 0.336 & - & 0.0001 & 0.004 & 0.012 & - \\
		MPLE$_{\beta}$ $(\beta=0.5)$ & 3 &  0.422 & 0.289 & 0.289 & - &       0.185 & 0.260 & 0.365 & - & 0.001 & 0.0004 & 0.004 & - \\
		MPLE$_{\beta}$ $(\beta=0.5)$ &4 & 0.421 & 0.289 & 0.153 & 0.137 &      0.185 & 0.261 & 0.330 & 0.422 &  0.0008 & 0.0004 & 0.0003 & 0.002 \\

		\hline 
	\end{tabular}
	\caption{Component parameter estimates for the SLC data. }
	\label{table 3}
\end{table}

The SLC dataset was originally composed of three clusters (representing the three genotypes). The histogram of the data shows three possible modes for the three probable clusters. However, it is observed that although the first two clusters (around the first two modes in the histogram) are approximately symmetric and bell-shaped in nature, the third cluster appears to significantly deviate from symmetry with a very long right tail. This leads to the discovery of only one significant real mode by the method of maximum likelihood, together with an almost invisible (and incorrect) second mode, and an entirely inaccurate third mode which is pushed way to the right to accommodate some very large observations on the right tail. The FE method possibly identifies the second mode, but it is far too tentative and diffused, perhaps due to its closeness with the first mode. The third mode is not at all discernible in the figure in this case, a consequence of the large estimated variance for the third component. The FE solution also appears to be substantially affected by the very large observations on the right tail.  The MPLE$_\beta$ provides a much more improved fit compared to the previous two. The first two modes are very accurately determined with suitable separation. The estimated third mode does not fully match the observed third mode, possibly because of the skewed pattern in the third cluster. However, unlike previous two fits, this fit clearly discounts the effect of the very large outliers to the right.    Observing the skewed third cluster in the data which is representing a model misspecification, a $4$-component normal mixture model has also been fitted using the MPLE$_{\beta}$ method (with $\beta = 0.5$) in order to assess whether this model can improve the fit. The three modes are now successfully and accurately recognized by the MPLE$_{\beta}$ method with the fourth fitted cluster pooling the skewed and misspecified part in the overall data. In this case also the large outliers are clearly discounted. In an overall sense, it can be concluded that the MPLE$_{\beta}$ method (with both $3$ and $4$ component normal mixture models) has provided substantially improved fits to these data compared to the the maximum likelihood and FE methods.

\subsection{Multivariate Case}
We now describe the utility of our method is going to be described by applying it on a multivariate dataset, namely, the Thyroid Gland Data. This dataset is among one of several databases in the Thyroid Disease Dataset 
of the UCI Machine Learning Repository (available at this \href{https://archive.ics.uci.edu/ml/datasets/thyroid+disease}{web address} and the R package $\sf{mclust}$ \cite{mclust}). The data provide information on $215$ patients about the laboratory test outcomes of five medical attributes. These attributes are \textit{(i)} $T3$-resin uptake test (in percentage, RT3U), \textit{(ii)} Total Serum thyroxin as measured by the isotopic displacement method (T4), \textit{(iii)} Total serum triiodothyronine as measured by radioimmuno assay (T3), \textit{(iv)} Basal thyroid-stimulating hormone (TSH) as measured by radioimmuno assay (TSH) and \textit{(v)} Maximal absolute difference of TSH value after injection of 200 micro grams of thyrotropin releasing hormone as compared to the basal value (DTSH). 
\begin{table}[!h]
	\centering 
	\small 
	\begin{tabular}{c c c c c c c c  c c c c c c c c c} 
		\hline
		Component&       \multicolumn{2}{c}{MLE} &  \multicolumn{2}{c}{MPLE$_{\beta}$} \\ [1ex] 
		\hline
		Means & Original & Contaminated & Original & Contaminated \\
$\hat{\mu}_1$	&	93.194	&	93.190	&	95.781	&	95.786	\\
	&	17.019	&	17.019	&	15.993	&	15.993	\\
	&	4.161	&	4.167	&	3.654	&	3.658	\\
	&	0.975	&	0.982	&	0.953	&	0.953	\\
	&	-0.047	&	-0.047	&	-0.049	&	-0.044	\\
	&		&		&		&		\\
$\hat{\mu}_2$	&	110.908	&	111	&	110.414	&	110.411	\\
	&	9.156	&	9.132	&	9.006	&	9.007	\\
	&	1.725	&	1.724	&	1.69	&	1.69	\\
	&	1.324	&	1.327	&	1.246	&	1.244	\\
	&	2.582	&	2.711	&	2.388	&	2.383	\\
	&		&		&		&		\\
$\hat{\mu}_3$	&	124.577	&	134.039	&	126.936	&	124.291	\\
	&	3.635	&	12.696	&	2.925	&	3.833	\\
	&	1.031	&	6.418	&	0.92	&	1.047	\\
	&	14.677	&	30.864	&	12.608	&	10.697	\\
	&	19.596	&	33.601	&	19.173	&	17.727	\\

		\hline

		\hline 
	\end{tabular}
	\caption{Component mean estimates for Thyroid Gland Data (MPLE$_{\beta}$ method with $\beta=0.3$). }
	\label{table 2}
	
\end{table}

The data also reveal the actual thyroidal state of these $215$ patients, i.e., whether they are suffering from euthyroidism (normal thyroid gland function),
hypothyroidism (underactive thyroid not producing enough thyroid hormone) or hyperthyroidism
(overactive thyroid producing and secreting excessive amounts of the free thyroid hormones T3
and/or thyroxine T4). Thus, we can fit a $3$-component normal mixture model (with $5$-dimensions) using the MPLE$_{\beta}$ method. Here the original cluster sizes are not equal and two of the original cluster sizes are only $30$ and $35$ while the data dimension is $5$. For this relatively larger $\frac{n}{p}$ value, the DPD estimates for each clusters need a stable starting value. We thus used S-estimates of location and scale (Lopuha$\ddot{a}$ $(1989)$ \cite{lopuhaa89}) in this regard. An exploratory analysis of these data appear to indicate that there are no major outliers in the dataset. This is also suggested by the similarity of the component mean estimates in case of the MPLE$_{\beta}$ (with $\beta = 0.3$) and the nonrobust maximum likelihood estimates in Table \ref{table 2}. To establish the outlier stability of the MPLE$_{\beta}$ method, we contaminate the original data artificially with $10$ additional points which are discrepant in comparison with the original data and can be viewed as outliers. These contaminating observations are listed in the Appendix \ref{appenC}. The component mean estimates of this artificially contaminated dataset by the maximum $\beta$-likelihood method (with $\beta=0.3$) and usual likelihood based method are also presented in Table \ref{table 2}. The stability of the MPLE$_\beta$ method is immediately observed in the minimal shifts in the  component mean estimates which obviously cannot be claimed for the maximum likelihood estimates. The variation in the estimates of the third cluster is higher than those of the other two, but here also the MPLE$_\beta$ estimator is far more stable then the MLE. The superiority of the MPLE$_\beta$ approach over the ordinary likelihood version is quite apparent, at least as far as the evidence of this example. The covariance matrix estimates for the original data and the artificially contaminated data for both likelihood based and MPLE$_{\beta}$ algorithms are presented in the Appendix \ref{appenD}. The superiority of the MPLE$_\beta$ approach over the ordinary likelihood version can again be observed in terms of greater stability of the estimated covariance matrix elements and greater sign consistency of the same.

\section{Concluding Remarks}
\label{con}
We have derived two important theoretical properties of the MPLE$_{\beta}$, namely, existence and weak consistency, for the MVN mixture models under appropriate theoretical conditions. Modern empirical process theory has been utilized to establish the consistency of the aforesaid estimators. The asymptotic distribution and convergence rates of these estimators are yet to be studied (probably requires the assumption of certain identifiability conditions) which we hope to develop in our future research. Two real data examples have also been presented to illustrate the practical utility of the MPLE$_\beta$ method.
\\
\\
 \textbf{Acknowledgement:} 
The research of AG is partially supported by the INSPIRE Faculty Research Grant from Department of Science and Technology, Government of India. The research of AB is supported by the Technology Innovation Hub at Indian Statistical Institute, Kolkata under
Grant NMICPS/006/MD/2020-21 of Department of Science and Technology, Government of India, dated 16.10.2020.
\begin{appendices}
	
	\section{Required Lemmas}
\label{appenA}	

\begin{lemma}\label{baz1}
	Consider the set-up of Theorem \ref{THM:Existence}, its proof and assume that $\pi_{j}>0$ for $j=1,\cdots,k$ in Equation (\ref{Eq19}). Then, we have the following results under the ER and NS constraints.
	\begin{enumerate}
		\item $g=k$ in Equation (\ref{Eq20}) for the mean sequence in the proof of Theorem \ref{THM:Existence}.
		\item The dispersion sequence $\{\boldsymbol{\Sigma}^{r}_{1},\boldsymbol{\Sigma}^{r}_{2},...,\boldsymbol{\Sigma}^{r}_{k}\}$, from the proof of Theorem \ref{THM:Existence}, only satisfies (\ref{Eq21}) (and not (\ref{Eq22}) or (\ref{Eq23})).
	\end{enumerate}
 \end{lemma}
\begin{proof}
	To prove the lemma, we need the following inequalities which seem to hold trivially from the definitions of $M_r$ and $m_r$.
	\begin{enumerate}
		\item[\textbf{I1}] For $1 \leq j \leq k$ and $r \in \mathbb{N}$,
		\begin{align*}
		m_{r}^{p}\leq |\boldsymbol{\Sigma}^{r}_{j}|\leq M_{r}^{p}.
		\end{align*}
		\item[\textbf{I2}] For $1 \leq j \leq k$ and $r \in \mathbb{N}$,
		\begin{align*}
		(\boldsymbol{X}-\boldsymbol{\mu}^{r}_{j})^{'}(\boldsymbol{\Sigma}^{r}_{j})^{-1}(\boldsymbol{X}-\boldsymbol{\mu}^{r}_{j})\geq M^{-1}_{r}||\boldsymbol{X}-\boldsymbol{\mu}^{r}_{j}||^{2}.
		\end{align*}
		\item[\textbf{I3}] For $1 \leq j \leq k$ and $r \in \mathbb{N}$,
		\begin{align*}
		\frac{1}{1+\beta}\int \phi_{p}^{1+\beta}(\boldsymbol{x},\boldsymbol{\mu}^{r}_{j},\boldsymbol{\Sigma}^{r}_{j})\;d\boldsymbol{x}&=\frac{1}{(2\pi)^{\frac{p\beta}{2}}|\boldsymbol{\Sigma}^{r}_{j}|^{\frac{\beta}{2}}(1+\beta)^{\frac{p+2}{2}}}\\
		&\geq \frac{1}{(2\pi)^{\frac{p\beta}{2}}M_{r}^{\frac{p\beta}{2}}(1+\beta)^{\frac{p+2}{2}}}.
		\end{align*}
	\end{enumerate}
	
	Using the above inequalities, we have,
	\begin{equation}
	\centering 
	\label{Eq24}
	L_{\beta}(\boldsymbol{\theta}^{r},P)\leq E_{P}\left[\sum_{j=1}^{k}Z_{j}(\boldsymbol{X},\boldsymbol{\theta}^{r})\left[\log\;\pi^{r}_{j}+\frac{1}{\beta(2\pi)^{\frac{p\beta}{2}}m_{r}^{\frac{p\beta}{2}}}e^{-\frac{\beta M^{-1}_{r}}{2}||\boldsymbol{X}-\boldsymbol{\mu}^{r}_{j}||^{2}}-\frac{1}{(2\pi)^{\frac{p\beta}{2}}M_{r}^{\frac{p\beta}{2}}(1+\beta)^{\frac{p+2}{2}}}\right]\right].
	\end{equation}
	Let us first prove the second part of Lemma \ref{baz1} . Suppose, if possible, (\ref{Eq22}) holds. Then the eigenvalue ratio constraint implies,
	\begin{align*}
	m_{r}\geq \frac{M_{r}}{c}\rightarrow \infty.
	\end{align*}
	These would imply,
	\begin{align*}
	\underset{r \rightarrow \infty}{\text{lim}}L_{\beta}(\boldsymbol{\theta}^{r},P)&\leq \underset{r \rightarrow \infty}{\text{lim}}E_{P}\left[\sum_{j=1}^{k}Z_{j}(\boldsymbol{X},\boldsymbol{\theta}^{r})\left[\log\;\pi^{r}_{j}+\frac{1}{\beta(2\pi)^{\frac{p\beta}{2}}m_{r}^{\frac{p\beta}{2}}}e^{-\frac{\beta M^{-1}_{r}}{2}||\boldsymbol{X}-\boldsymbol{\mu}^{r}_{j}||^{2}}-\frac{1}{(2\pi)^{\frac{p\beta}{2}}M_{r}^{\frac{p\beta}{2}}(1+\beta)^{\frac{p+2}{2}}}\right]\right]\\
	&\leq \underset{r \rightarrow \infty}{\text{lim}}E_{P}\left[\sum_{j=1}^{k}\log\;\pi^{r}_{j}\right]<0
	\end{align*}
	which contradicts (\ref{Eq18}).
	Now, let us assume that (\ref{Eq23}) holds. But this contradicts the non-singularity constraint. Hence the dispersion sequence can only satisfy the condition in Equation (\ref{Eq21}), and not the conditions (\ref{Eq22}) or (\ref{Eq23}).
	
	To prove the first part of the Lemma (i.e., $g=k$ in Equation (\ref{Eq20})), let us observe that if $g=0$ then, $||\boldsymbol{\mu}^{r}_{j}||\rightarrow \infty$ and thus $e^{-||\boldsymbol{X}-\boldsymbol{\mu}^{r}_{j}||^{2}}\rightarrow 0$ for all $1\leq j\leq k$. Hence, (\ref{Eq24}) again implies,
	$\underset{r\rightarrow\infty}{\text{lim}}L_{\beta}(\boldsymbol{\theta}^{r},P)\leq 0.$
	which contradicts (\ref{Eq18}). Hence $g>0$.
	
	Next let us assume that, $1\leq g< k$. Then {\it bounded convergence theorem}  implies,
	\begin{align}
	\label{Eq25}
	E_{P}\left(\sum_{j=g+1}^{k}Z_{j}(\boldsymbol{X},\boldsymbol{\theta}^{r})\right)\rightarrow 0.
	\end{align}
	Now,
	\begin{align*}
	\underset{r\rightarrow\infty}{\text{limsup}}\;L_{\beta}(\boldsymbol{\theta}^{r},P)&=\underset{r\rightarrow\infty}{\text{limsup}}\;E_{P}\left[\sum_{j=1}^{k}Z_{j}(\boldsymbol{X},\boldsymbol{\theta}^{r})\left[\log\;\pi^{r}_{j}+\frac{1}{\beta}\phi_{p}^{\beta}(\boldsymbol{X},\boldsymbol{\mu}^{r}_{j},\boldsymbol{\Sigma}^{r}_{j})-\frac{1}{1+\beta}\int \phi_{p}^{1+\beta}(\boldsymbol{x},\boldsymbol{\mu}^{r}_{j},\boldsymbol{\Sigma}^{r}_{j}) \;d\boldsymbol{x}\right]\right]\\
	&\leq \underset{r\rightarrow\infty}{\text{limsup}}\;E_{P}\left[\sum_{j=1}^{g}Z_{j}(\boldsymbol{X},\boldsymbol{\theta}^{r})\left[\log\;\pi^{r}_{j}+\frac{1}{\beta}\phi_{p}^{\beta}(\boldsymbol{X},\boldsymbol{\mu}^{r}_{j},\boldsymbol{\Sigma}^{r}_{j})-\frac{1}{1+\beta}\int \phi_{p}^{1+\beta}(\boldsymbol{x},\boldsymbol{\mu}^{r}_{j},\boldsymbol{\Sigma}^{r}_{j}) \;d\boldsymbol{x}\right]\right]\\
	&+\underset{r\rightarrow\infty}{\text{limsup}}\;E_{P}\left[\sum_{j=g+1}^{k}Z_{j}(\boldsymbol{X},\boldsymbol{\theta}^{r})\left[\log\;\pi^{r}_{j}+\frac{1}{\beta}\phi_{p}^{\beta}(\boldsymbol{X},\boldsymbol{\mu}^{r}_{j},\boldsymbol{\Sigma}^{r}_{j})-\frac{1}{1+\beta}\int \phi_{p}^{1+\beta}(\boldsymbol{x},\boldsymbol{\mu}^{r}_{j},\boldsymbol{\Sigma}^{r}_{j}) \;d\boldsymbol{x}\right]\right].
	\end{align*}
	The second term in the right hand side of the above inequality less than equal to $0$ due to (\ref{Eq25}). Hence,
	\begin{align*}
	\underset{r\rightarrow\infty}{\text{limsup}}\;L_{\beta}(\boldsymbol{\theta}^{r},P)
	&\leq\underset{r\rightarrow\infty}{\text{limsup}}\;E_{P}\left[\sum_{j=1}^{g}Z_{j}(\boldsymbol{X},\boldsymbol{\theta}^{r})\left[\log\;\pi^{r}_{j}+\frac{1}{\beta}\phi_{p}^{\beta}(\boldsymbol{X},\boldsymbol{\mu}^{r}_{j},\boldsymbol{\Sigma}^{r}_{j})-\frac{1}{1+\beta}\int \phi_{p}^{1+\beta}(\boldsymbol{x},\boldsymbol{\mu}^{r}_{j},\boldsymbol{\Sigma}^{r}_{j}) \;d\boldsymbol{x}\right]\right]\\
	&=E_{P}\left[\sum_{j=1}^{g}Z_{j}(\boldsymbol{X},\boldsymbol{\theta}^{*})\left[\log\;\pi_{j}+\frac{1}{\beta}\phi_{p}^{\beta}(\boldsymbol{X},\boldsymbol{\mu}_{j},\boldsymbol{\Sigma}_{j})-\frac{1}{1+\beta}\int \phi_{p}^{1+\beta}(\boldsymbol{x},\boldsymbol{\mu}_{j},\boldsymbol{\Sigma}_{j}) \;d\boldsymbol{x}\right]\right]\\
	\end{align*}
	where,
	\begin{align*}
	\boldsymbol{\theta}^{*}&= \left(\pi_{1}, \pi_{2},...., \pi_{g}, \boldsymbol{\mu}_{1}, \boldsymbol{\mu}_{2},...,\boldsymbol{\mu}_{g}, \boldsymbol{\Sigma}_{1}, \boldsymbol{\Sigma}_{2},..., \boldsymbol{\Sigma}_{g}\right)\\
	&=\underset{r\rightarrow\infty}{\text{lim}}\; \left(\pi^{r}_{1}, \pi^{r}_{2},...., \pi^{r}_{g}, \boldsymbol{\mu}^{r}_{1}, \boldsymbol{\mu}^{r}_{2},...,\boldsymbol{\mu}^{r}_{g}, \boldsymbol{\Sigma}^{r}_{1}, \boldsymbol{\Sigma}^{r}_{2},..., \boldsymbol{\Sigma}^{r}_{g}\right)
	\end{align*}
	and $\pi_{j}$, $\boldsymbol{\mu}_{j}$ and $\boldsymbol{\Sigma}_{j}$ are as in (\ref{Eq19}), (\ref{Eq20}) and (\ref{Eq21}) respectively. Let us observe that $\sum_{j=1}^{g}\pi_{j}<1$ due to the assumption that $\pi_{j}>0$ for all $1\leq j \leq k$. Motivated by this observation, we introduce the following standardized weights,
	\[
	\pi^{'}_{j}= 
	\begin{cases}
	\frac{\pi_{j}}{\sum_{j=1}^{g}\pi_{j}},& \text{for } 1 \leq j \leq g\\
	0, & \text{for } j>g.
	\end{cases}
	\]
	It is easy to observe that,
	\begin{enumerate}
		\item For all $1 \leq j \leq g$, $\log\;\pi_{j}<\log\;\pi^{'}_{j}.$
		\item This aforesaid modification keeps the orderings of the discriminant functions $\{D_{j}(\boldsymbol{X},\cdot):\;1 \leq j \leq k\}$ invariant so that values of the assignment functions $\{Z_{j}(\boldsymbol{X},\cdot):\;1 \leq j \leq k\}$ would not change.
	\end{enumerate}
	The aforesaid facts together imply,
	\begin{align*}
	\underset{r\rightarrow\infty}{\text{limsup}}\;L_{\beta}(\boldsymbol{\theta}^{r},P)&\leq E_{P}\left[\sum_{j=1}^{g}Z_{j}(\boldsymbol{X},\boldsymbol{\theta}^{*})\left[\log\;\pi_{j}+\frac{1}{\beta}\phi_{p}^{\beta}(\boldsymbol{X},\boldsymbol{\mu}_{j},\boldsymbol{\Sigma}_{j})-\frac{1}{1+\beta}\int \phi_{p}^{1+\beta}(\boldsymbol{x},\boldsymbol{\mu}_{j},\boldsymbol{\Sigma}_{j}) \;d\boldsymbol{x}]\right]\right]\\
	&< E_{P}\left[\sum_{j=1}^{g}Z_{j}(\boldsymbol{X},\boldsymbol{\theta}^{*'})\left[\log\;\pi^{'}_{j}+\frac{1}{\beta}\phi_{p}^{\beta}(\boldsymbol{X},\boldsymbol{\mu}_{j},\boldsymbol{\Sigma}_{j})-\frac{1}{1+\beta}\int \phi_{p}^{1+\beta}(\boldsymbol{x},\boldsymbol{\mu}_{j},\boldsymbol{\Sigma}_{j}) \;d\boldsymbol{x}]\right]\right]\\
	&=L_{\beta}(\boldsymbol{\theta}^{*'},P),
	\end{align*}
	where $\boldsymbol{\theta}^{*'}= \left(\pi^{'}_{1}, \pi^{'}_{2},...., \pi^{'}_{g}, \boldsymbol{\mu}_{1}, \boldsymbol{\mu}_{2},...,\boldsymbol{\mu}_{g}, \boldsymbol{\Sigma}_{1}, \boldsymbol{\Sigma}_{2},..., \boldsymbol{\Sigma}_{g}\right) \in \boldsymbol{\Theta}_{C}$. But this contradicts (\ref{Eq17}). Hence, $g=k$, completing the proof of the Lemma. 
\end{proof}

\begin{lemma}\label{boundedness}
	For the optimizers $\{\hat{\boldsymbol{\theta}}_{n}:\;n \in \mathbb{N}\}$  of the objective function given in (\ref{Eq4}), there exists a compact set $K \subset \boldsymbol{\Theta}_{C}$ such that $\hat{\boldsymbol{\theta}}_{n} \in K$ a.e. $[F]$ for all sufficiently large $n$ with probability $1$ under the ER and NS constraints and Assumption \ref{assump1}.
\end{lemma}

\begin{proof}
	The proof of this lemma proceeds along the same line as in the proof of the existence of the estimators. To prove our claim, it is enough to show that the sequence of the largest eigenvalues of all the dispersion matrices $M_{n}$ does not converge to $\infty$, the sequence of the smallest eigenvalues of all the dispersion matrices $m_{n}$ does not converge to $0$ and the estimated centers $\hat{\boldsymbol{\mu}}^{n}_{j}$ does not satisfy $||\hat{\boldsymbol{\mu}}^{n}_{j}||\rightarrow\infty$ for any $1\leq j\leq k$ a.e. $[F]$ as $n \rightarrow \infty$. To prove these, first let us observe that,
	\begin{align*}
	L_{\beta}(\hat{\boldsymbol{\theta}}_{n},F_{n})&=\underset{\boldsymbol{\Theta}_{C}}{\text{sup}} L_{\beta}(\boldsymbol{\theta} , F_{n})\\
	&\geq L_{\beta}(\boldsymbol{\theta}^{a},F_{n})\\
	&=\frac{1}{n\beta}\sum_{i=1}^{n}\phi^{\beta}_{p}(\boldsymbol{X}_{i},\boldsymbol{\mu^{0}_{1}},c^{'}\boldsymbol{\Sigma^{0}_{1}})-\frac{1}{1+\beta}\int \phi^{1+\beta}_{p}(\boldsymbol{x},0,c^{'}\boldsymbol{\Sigma^{0}_{1}})\;d\boldsymbol{x}.
	\end{align*}
	Now, using the same methodology to prove (\ref{Eq18}) in case of $P=F_{n}$ for sufficiently large sample sizes by virtue of SLLN (in Lemma \ref{baz1}), we can conclude that,
	\begin{align*}
	L_{\beta}(\hat{\boldsymbol{\theta}}_{n},F_{n})\geq 0
	\end{align*}
	for sufficiently large sample sizes. Now, we can show that the estimators $\hat{\boldsymbol{\theta}}_{n}$ are uniformly bounded for all large enough $n$ by an argument similar to the one used to proof Theorem \ref{THM:Existence}  (and Lemma \ref{baz1}).
\end{proof}
\newpage
\section{Contaminating Observations Added to the Thyroid Gland Data}
\label{appenC}
\begin{table}[!h]
	\centering 
	\small 
	\begin{tabular}{c c c c c c } 
		\hline
		Units&       \multicolumn{5}{c}{Contaminating Observations}  \\ [1ex] 
		\hline
		1	&	155.704	&	36.535	&	17.451	&	66.078	&	64.804	\\
		2	&	156.124	&	35.039	&	20.026	&	67.2	&	66.104	\\
		3	&	154.383	&	33.625	&	20.428	&	68.302	&	66.031	\\
		4	&	156.89	&	34.951	&	20.496	&	66.417	&	65.875	\\
		5	&	156.572	&	34.154	&	17.831	&	68.719	&	65.724	\\
		6	&	155.042	&	35.24	&	18.686	&	68.095	&	65.878	\\
		7	&	154.823	&	33.75	&	20.586	&	66.143	&	66.802	\\
		8	&	153.185	&	36.636	&	19.51	&	66.242	&	68.082	\\
		9	&	155.616	&	35.583	&	19.511	&	67.386	&	64.267	\\
		10	&	155.996	&	36.244	&	20.071	&	66.41	&	64.583	\\

		\hline 
	\end{tabular}
	\caption{Contaminating observations added to the Thyroid Gland Data. }
	\label{table 5}
	
\end{table}
\section{Component Covariance Matrix Estimates for the Thyroid Data}
\label{appenD}
\begin{table}[!h]
	\small
	\begin{tabular}{c c c c c c c c c c c c} 
		\hline
		Cluster &       \multicolumn{5}{c}{Original Data}  & \multicolumn{5}{c}{Contaminated Data}  \\ [1ex] 
		
		\hline
		1	&	218.051	&	-30.978	&	-21.142	&	-1.284	&	0.608	&	&	218.047	&	-30.973	&	-21.148	&	-1.281	&	0.604	\\
		&	-30.975	&	20.704	&	5.188	&	0.154	&	0.006	&	&	-30.978	&	20.704	&	5.182	&	0.154	&	0.007	\\
		&	-21.144	&	5.188	&	5.130	&	0.063	&	0.014	&	&	-21.142	&	5.188	&	5.125	&	0.067	&	0.011	\\
		&	-1.281	&	0.154	&	0.063	&	0.162	&	-0.03	&	&	-1.286	&	0.152	&	0.062	&	0.162	&	-0.031	\\
		&	0.604	&	0.006	&	0.011	&	-0.039	&	0.061	&	&	0.607	&	0.006	&	0.011	&	-0.035	&	0.061	\\
		&		&		&		&		&		&	&		&		&		&		&		\\
		2	&	77.36	&	9.764	&	2.166	&	-0.355	&	-0.129	&	&	77	&	9.469	&	2.126	&	-0.325	&	0.783	\\
		&	9.764	&	6.356	&	0.515	&	-0.04	&	-1.283	&	&	9.469	&	6.326	&	0.514	&	-0.047	&	-1.493	\\
		&	2.166	&	0.515	&	0.265	&	-0.005	&	0.003	&	&	2.126	&	0.514	&	0.263	&	-0.006	&	-0.009	\\
		&	-0.355	&	-0.04	&	-0.005	&	0.245	&	0.106	&	&	-0.325	&	-0.047	&	-0.006	&	0.243	&	0.138	\\
		&	-0.129	&	-1.283	&	0.003	&	0.106	&	3.784	&	&	0.783	&	-1.493	&	-0.009	&	0.138	&	5.037	\\
		&		&		&		&		&		&	&		&		&		&		&		\\
		3	&	72.254	&	-5.453	&	-0.858	&	-12.438	&	-58.142	&	&	248.724	&	203.876	&	119.329	&	317.76	&	248.581	\\
		&	-5.453	&	4.118	&	0.807	&	-14.686	&	11.245	&	&	203.876	&	219.471	&	126.209	&	341.553	&	320.974	\\
		&	-0.858	&	0.807	&	0.285	&	-2.802	&	3.722	&	&	119.329	&	126.209	&	73.552	&	201.171	&	183.511	\\
		&	-12.438	&	-14.686	&	-2.802	&	153.898	&	-6.269	&	&	317.76	&	341.553	&	201.171	&	669.271	&	490.303	\\
		&	-58.142	&	11.245	&	3.722	&	-6.269	&	245.921	&	&	248.581	&	320.974	&	183.511	&	490.303	&	628.838	\\

		\hline
	\end{tabular}
	\caption{Component covariance matrix estimates for the Thyroid data in case of the maximum likelihood estimation.}
	\label{table3.1}
\end{table}

\begin{table}[!h]
	\small
	\begin{tabular}{c c c c c c c c c c c c} 
		\hline
		Cluster &       \multicolumn{5}{c}{Original Data}  & \multicolumn{5}{c}{Contaminated Data}  \\ [1ex] 
		
		\hline
		1	&	155.261	&	-28.989	&	-17.165	&	-0.107	&	0.002	&	&	155.263	&	-28.983	&	-17.165	&	-0.105	&	0.002	\\
		&	-28.989	&	24.782	&	5.454	&	-0.233	&	0.251	&	&	-28.984	&	24.782	&	5.449	&	-0.237	&	0.256	\\
		&	-17.165	&	5.454	&	3.845	&	-0.058	&	0.082	&	&	-17.162	&	5.453	&	3.841	&	-0.052	&	0.086	\\
		&	-0.102	&	-0.233	&	-0.058	&	0.167	&	-0.021	&	&	-0.105	&	-0.231	&	-0.053	&	0.162	&	-0.02	\\
		&	0.002	&	0.251	&	0.082	&	-0.02	&	0.057	&	&	0.002	&	0.252	&	0.084	&	-0.021	&	0.055	\\
		&		&		&		&		&		&	&		&		&		&		&		\\
		2	&	66.207	&	5.769	&	1.48	&	-0.034	&	2.163	&	&	65.676	&	5.728	&	1.47	&	-0.033	&	2.136	\\
		&	5.769	&	4.493	&	0.393	&	-0.03	&	-0.452	&	&	5.728	&	4.456	&	0.391	&	-0.029	&	-0.436	\\
		&	1.48	&	0.393	&	0.22	&	-0.002	&	0.099	&	&	1.47	&	0.391	&	0.218	&	-0.002	&	0.1	\\
		&	-0.034	&	-0.03	&	-0.002	&	0.202	&	0.041	&	&	-0.033	&	-0.029	&	-0.002	&	0.199	&	0.039	\\
		&	2.163	&	-0.452	&	0.099	&	0.041	&	3.232	&	&	2.136	&	-0.436	&	0.1	&	0.039	&	3.169	\\
		&		&		&		&		&		&	&		&		&		&		&		\\
		3	&	62.309	&	-3.334	&	-1.324	&	-5.678	&	-109.539	&	&	98.298	&	-13.328	&	-2.083	&	21.443	&	-69.5	\\
		&	-3.334	&	2.852	&	0.734	&	-4.84	&	22.356	&	&	-13.328	&	6.382	&	1.236	&	-12.788	&	12.732	\\
		&	-1.324	&	0.734	&	0.311	&	-1.418	&	6.925	&	&	-2.083	&	1.236	&	0.388	&	-2.276	&	4.817	\\
		&	-5.678	&	-4.84	&	-1.418	&	31.978	&	-10.72	&	&	21.443	&	-12.788	&	-2.276	&	52.686	&	-0.003	\\
		&	-109.539	&	22.356	&	6.925	&	-10.72	&	386.56	&	&	-69.5	&	12.732	&	4.817	&	-0.003	&	321.714	\\

		\hline
	\end{tabular}
	\caption{Component covariance matrix estimates for the Thyroid data in case of the minimum DPD estimation.}
	\label{table3.2}
\end{table}

\end{appendices}

 \end{document}